\documentclass[reqno]{amsart}
\usepackage{amsmath,amsthm, amsbsy, amsfonts, amssymb}
\usepackage{euscript, mathrsfs, stmaryrd, mathbbol} 
\usepackage{mathtools} 
\usepackage{graphicx, amscd, tikz-cd, tikz} 
\usetikzlibrary[matrix, arrows, arrows.meta, calc, decorations.pathmorphing, backgrounds, positioning, fit, petri, cd]
\usepackage{setspace}
\usepackage{hyperref, url}
\usepackage{xcolor}
\newcommand{\C}{\mathbb{C}} 
\newcommand{\Q}{\mathbb{Q}} 
\newcommand{\Z}{\mathbb{Z}}  

\newcommand{\abs}[1]{\lvert#1\rvert}
\newcommand{\eu}[1]{\EuScript{#1}}

\newcommand{\mbb}[1]{\mathbb{#1}}
\newcommand{\mc}[1]{\mathcal{#1}}

\newcommand{\msf}[1]{\mathsf{#1}}
\newcommand{\op}[1]{\operatorname{#1}}

\swapnumbers
\newtheorem{theorem}{Theorem}[section]
\newtheorem{lemma}[theorem]{Lemma}

\theoremstyle{definition}
\newtheorem{definition}[theorem]{Definition}
\newtheorem{example}[theorem]{Example}
\newtheorem{topic}[theorem]{}
\theoremstyle{remark}
\newtheorem{remark}[theorem]{Remark}
\newcommand{\bstar}{b_*}

\newcommand{\edge}{\msf{Edg}}
\newcommand{\ip}[2]{\langle #1 , #2 \rangle }
\newcommand{\pentagon}{\msf{Pen}}
\newcommand{\XFF}{X}
\newcommand{\wstar}{w_*}
\begin{document}
\title[Petersen Graph and monodromy of the 27 lines]{Petersen graph and monodromy of the 27 lines on the Clebsch surface}
\author{Tathagata Basak}
\address{Department of Mathematics\\Iowa State University, \\Ames, IA 50011}
\email{tathagat@iastate.edu}
%
\thanks{ Supported by Simons Foundation Collaboration Grant 637005.}
%
%
\subjclass[2020]{Primary: 14D05, 14F35; Secondary: 14J26, 20F55.}
%

%
\maketitle
\begin{abstract}
Let $G$ be the orbifold fundamental group of the moduli space of smooth cubic surfaces
$\mc{M}_{\msf{sm}}$ 
in $\mathbb{P}^3_{\C}$
with base point at the Clebsch surface $X_{\mathbb{1}}$.
The image of the monodromy action $G \to \lbrace \text{Permutations of $27$ lines on $X_{\mathbb{1}}$} \rbrace$
  is famously 
the Weyl group of type $E_6$. Here we give a
description of this monodromy action in terms of the Petersen graph $\eu{P}$
by working out the action of ten explicit generators of $G$ by elementary calculation.
These ten generators were found in joint work with 
Allcock and Looijenga while studying the description of $\mc{M}_{\msf{sm}}$ 
as a discriminant complement in a complex $4$-ball quotient.
\end{abstract}
%
%
%
\section{introduction}
Cubic surfaces in $\mathbb{P}^3$ with the configuration of $27$ lines on them 
are one of the most studied examples in algebraic geometry going back to Cayley-Salmon (1849).
There is an enormous amount of literature on them.
For two detailed modern treatments, see  \cite{H} or chapter 9 of \cite{D}.
\par
Let $\eu{P}$ denote the Petersen graph; see figure \ref{figure-Petersen}.
Let $\mc{M}_{\msf{st}}$ and $\mc{M}_{\msf{sm}}$ denote the moduli space of stable 
and smooth cubic surfaces respectively. 
Let $X_{\mathbb{1}}$ denote the Clebsch cubic surface; 
this is  the unique smooth cubic surface with $S_5$ symmetry.
Let $G$ be the orbifold fundamental group of $\mc{M}_{\msf{sm}}$ with base-point
at $X_{\mathbb{1}}$. 
By \cite{act:cubic_surfaces},  $\mc{M}_{\msf{st}}$
is a complex $4$-ball quotient.
While studying this ball quotient, we found ten natural generators 
$\lbrace g_{A} \colon A \in \eu{P} \rbrace$
of $G$ (see \ref{def-meridian-in-Sylvester-family}) 
that yield a nice presentation of $G$ as a quotient of the Artin group of 
$\eu{P}$ (this is joint work  in progress \cite{ABL}).
Let $\mc{L}(X_{\mathbb{1}})$ be the set of $27$ lines on $X_{\mathbb{1}}$.
The elements of $\mc{L}(X_{\mathbb{1}})$ can be naturally labeled by the $15$ 
edges and the $12$ pentagons in $\eu{P}$. 
We fix such a labeling in \ref{def-27lines}.
If $S$ is an edge of $\eu{P}$ or a pentagon in $\eu{P}$, then let $L_S$ temporarily denote the 
corresponding  line on $X_{\mathbb{1}}$. 
If we move $X_{\mathbb{1}}$ along a loop in the moduli space $\mc{M}_{\msf{sm}}$, the
$27$ lines on $X_{\mathbb{1}}$ get permuted by the time we get back to $X_{\mathbb{1}}$.
This defines the famous monodromy representation of $G$ on the set $\mc{L}(X_{\mathbb{1}})$.
Our job here is to explicitly describe this monodromy action.
With the above notation, 
the action of the ten generators $g_A$ on the set $\mc{L}(X_{\mathbb{1}})$
can be described in terms of $\eu{P}$ as follows (also see Figure \ref{fig-six-Petersens}):
\begin{theorem}[See \ref{theorem-monodromy} for a more detailed statement]
\label{theorem-monodromy-and-Petersen}
Let $A$ be a vertex of the Petersen graph $\eu{P}$. 
Let $\lbrace E_i \colon i \in \Z/6\Z \rbrace$ be the edges 
of the hexagon  in $\eu{P}$ that is not connected to $A$,  in cyclic order.
So $E_{i + 3}$ is the edge opposite $E_i$. 
Let $\pentagon(A, E_{i})$ be the unique pentagon in $\eu{P}$ containing
$A$ and $E_i$.
Then the monodromy action of $g_A$ on $\mc{L}(X_{\mathbb{1}})$
exchanges the lines $L_{E_i}$ and $L_{\pentagon(A, E_{i + 3})}$
for $i \in \Z/6\Z$. These twelve lines form a double-six configuration.
The remaining fifteen lines are fixed by $g_A$.
\end{theorem}
\begin{figure}
\centerline{
\begin{tikzpicture}[
    scale=.64, 
    transform shape,
    vertex/.style={circle, draw, minimum size=1.8em, inner sep=0pt, font=\large},
    thick
  ]
  \foreach \i/\label in {1/13, 2/24, 3/35, 4/41, 5/52} {
    \node[vertex] (v\i) at ({90-(\i-1)*72}:3cm) {\label};
  }

  \foreach \i/\label in {1/45, 2/51, 3/12, 4/23, 5/34} {
    \node[vertex] (u\i) at ({90-(\i-1)*72}:1.5cm) {\label};
  }

  \draw (v1) -- (v2) -- (v3) -- (v4) -- (v5) -- (v1);

  \draw (u1) -- (u3) -- (u5) -- (u2) -- (u4) -- (u1);

  \foreach \i in {1,...,5} {
    \draw (v\i) -- (u\i);
  }

\end{tikzpicture}
}
\caption{The Petersen graph $\eu{P}$. The vertices are two element subsets of 
$\lbrace 1, 2, 3, 4, 5 \rbrace$, joined by an edge if disjoint.}
\label{figure-Petersen}
\end{figure}
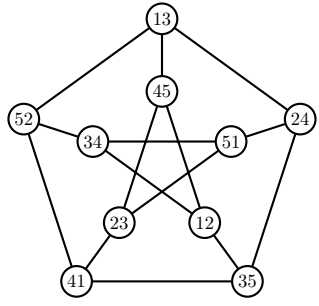
Now let us describe the objects involved and the plan for our calculation
in more detail.
A cubic surface in $\mbb{P}^3 = \mbb{P}^3_{\C}$ is 
the zero set of a cubic form (homogeneous cubic polynomial) in four variables.
A cubic surface is \textit{stable} if it has at most ordinary double point singularities.
Let $V = \C^4$. Let $\op{sym}^3( V^*) \simeq \mbb{A}^{20}$ be the space of 
homogeneous cubic forms in four variables. Let
 $\op{sym}^3( V^*)_{\msf{st}}$ and $\op{sym}^3( V^*)_{\msf{sm}}$ be
 its subsets consisting of forms that define stable and smooth
 hypersurfaces respectively. Quotienting these spaces by the
  $\op{GL}(V)$ action one obtains the moduli spaces
$\mc{M}_{\msf{st}} $ of stable cubic surfaces and $\mc{M}_{\msf{sm}}$ of smooth cubic surfaces.
The moduli space $\mc{M}_{\msf{st}}$ is a four dimensional complex orbifold
in which the non-smooth surfaces form a divisor $\Delta$ called the \textit{discriminant divisor}.
One has $\mc{M}_{\msf{sm}} = \mc{M}_{\msf{st}} - \Delta$.
 \par
To describe $\mc{M}_{\msf{st}}$ as a ball quotient following \cite{act:cubic_surfaces},
start with the ring $\mc{E} = \Z[\zeta_6]$ of Eisenstein integers where
$\zeta_6 = e^{2 \pi i /6}$. An $\mc{E}$-lattice $K$ means a free $\mc{E}$-module of finite rank 
with a nonsingular hermitian form $\ip{\;}{\;}: K \times K \to \Q[\zeta_6]$. Write 
$\C K = \C \otimes_{\mc{E}} K$.
Let $ v^2 :=  \ip{v}{v}$ be the \textit{norm} of $v$.
The \textit{dual of $K$}, denoted 
$K^*$, means the set of all $x \in \C K$ such that 
$ \ip{x}{K} \subseteq \eu{E}$.
 Let $L $ be the unique self-dual $\mc{E}$-lattice of signature $(4, 1)$.
 \par
The set of complex lines of negative norm vectors in $\C L$, denoted by 
$\mbb{B}(L) \subseteq \mbb{P}( \C  L)$,  is topologically a complex $4$-ball.
The group $P U(\C L) \simeq PU(4,  1)$ acts faithfully on $\mbb{B}(L)$
preserving a unique up-to scaling metric of negative curvature.
The ball $\mbb{B}(L)$ with this metric is called the complex hyperbolic space of dimension $4$.
The elements of 
\[
L(1) := \lbrace v \in L \colon v^2 = 1 \rbrace
\] 
are called the \textit{short roots} (of $L$).
The group $\op{Aut}(L) \subseteq U(4, 1)$ contains complex $\zeta_6$-reflections in the 
short roots.
A  complex $\zeta_6$-reflection $R_v^{\zeta_6}$ in $v \in L(1)$  has order $6$; it 
pointwise fixes the hyperplane $v^{\bot}$
and multiplies by $\zeta_6$ in the complex line $\C v$. The hyperplane $v^{\bot}$ or the sub-ball
$\mbb{B}(v^{\bot})$ in $\mbb{B}(L)$ is called the \textit{mirror} of the reflection $R_v^{\zeta_6}$.
 Let $\Gamma \subseteq P \op{Aut}(L)$ be the subgroup generated by the $\zeta_6$-reflections in the
 short roots.
 Let $\eu{H} \subseteq \mbb{B}(L)$ be the union of the mirrors of the reflections in $\Gamma$
 and let $\mathbb{B}(L)^{\circ} = \mathbb{B}(L) - \eu{H}$.
 The image $\Gamma \backslash \eu{H}$ of the mirrors in the orbifold $ \Gamma \backslash \mbb{B}(L)$
is an irreducible divisor and $ \Gamma \backslash \mbb{B}(L)^{\circ} $ is the complement of this divisor.
In \cite{act:cubic_surfaces}, Allcock-Carleson-Toledo, proved the following theorem: 

\begin{theorem}[\cite{act:cubic_surfaces}]
\label{theorem-ACT}
One has an isomorphism of complex orbifolds $\mc{M}_{\msf{st}} \simeq \Gamma \backslash \mbb{B}(L)$ 
and this restricts to an isomorphism
$\mc{M}_{\msf{sm}} \simeq \Gamma \backslash \mbb{B}(L)^{\circ}$.
\end{theorem}
Let $\tau$ be a point in $\mbb{B}(L)$ whose image in $\Gamma \backslash \mbb{B}(L)$ is 
the Clebsch surface $X_{\mbb{1}}$. Let 
\[
G = \pi_1^{\op{orb}}(\Gamma \backslash \mbb{B}(L)^{\circ} , \tau)
 = \pi_1^{\op{orb}}( \mc{M}_{\msf{sm}}, X_{\mathbb{1}}).
\]
It is well known that the orbifold fundamental group $G$
has Artin group-like presentations \cite{Li, Lo}.
Investigating the  ball quotient $\Gamma \backslash \mbb{B}(L)$ near the point $\tau$
led us to a nice presentation of 
$\pi_1^{\op{orb}}(\Gamma \backslash \mbb{B}(L)^{\circ} , \tau)$.
To describe it, 
let $\eu{P} = K(2, 5)$ be the Petersen graph whose vertices are the two element subsets of 
$\lbrace 1, 2, 3, 4, 5 \rbrace$
and where two vertices are joined if and only if the two subsets are disjoint.
For each vertex $\lbrace i, j \rbrace$ of $\eu{P}$, we'll specify a short root $\alpha_{i j}$ in $L$.
Let $R_{i j} = R_{\alpha_{i j}}^{\zeta_6}$.
It turns out that the ten mirrors $\alpha_{i j}^{\bot}$ of $\Gamma$ are precisely the
mirrors of $\Gamma$ that are closest to $\tau$ and equidistant from 
$\tau$ and $\Gamma$ is generated by  the ten complex reflections $R_{i j}$.
For each vertex $\lbrace i, j \rbrace$ of $\eu{P}$, let
$(\mu_{i j}, R_{i j})$ be the  canonical element of 
$\pi_1^{\op{orb}}(\Gamma \backslash \mbb{B}(L)^{\circ} , \tau)$
 that is represented by a path $\mu_{i j}$ that
 goes around the mirror $\alpha_{i j}^{\bot}$. These canonical loops are called 
 \textit{meridians} (in the ball model)  around $\alpha_{i j}^{\bot}$; they are carefully defined in 
 \cite{AB2} section 3.
In the forthcoming article \cite{ABL} we prove that the ten meridians $(\mu_{i j}, R_{i j})$,
(or rather their explicit incarnations $g_{i j}$ in the Sylvester family defined in \ref{def-meridian-in-Sylvester-family}) 
generate 
the fundamental group $G$ 
and that this leads to a presentation of $G$ as a quotient of the Artin group of $\eu{P}$.
To the best of our knowledge, \cite{ABL} is the first work that highlights
the importance of the 
subgroup $\op{Aut}(\eu{P}) = S_5 \subseteq W(E_6)$ 
in understanding the whole moduli space $\mc{M}_{\msf{sm}}$.
\par
In this paper, we explicitly describe the monodromy action of $G$ on the $27$ lines 
$\mc{L}(X_{\mbb{1}})$ of the Clebsch surface $X_{\mbb{1}}$.
This is possible because of the following facts: 
an open subset of $\mc{M}_{\msf{st}}$ can be represented 
by a four parameter family $\mc{X}$ of cubic forms (called the Sylvester family) 
such that $\mc{X} \to \mc{M}_{\msf{st}}$ is a $S_5$-cover. 
It turns out that, locally near  $\tau$, the map $\mbb{B} (L) \to \eu{M}_{\msf{st}}$ is also an $S_5$-cover. 
So there is a open neighborhood of 
$\tau$ in $\mbb{B}(L)$ such that every point in it can be uniquely represented by a form in the Sylvester family
and this identification is $S_5$ equivariant. Using this, we write down
ten explicit loops $g_{i j}$'s in the space $\mc{X}$ based at $X_{\mathbb{1}}$
that correspond to the $(\mu_{i j}, R_{i j})$'s and then work out how the $27$ lines on $X_{\mathbb{1}}$
move as we go around these ten loops. 
This yields our main theorem \ref{theorem-monodromy-and-Petersen}
or rather its more detailed version \ref{theorem-monodromy}. In  theorem \ref{theorem-monodromy-and-Petersen}
we write $g_A = g_{i j}$ if $A = \lbrace i , j \rbrace$ is a vertex of $\eu{P}$.
\par
We describe easy to remember names of the $27$ lines on $X_{\mathbb{1}}$ and 
the action of the ten generators $g_{i j}$ are written down as ten explicit permutations of $\mc{L}(X_{\mbb{1}})$.
One verifies easily that these ten permutations generate
the full Weyl group $W(E_6)$; see \cite{HR}. 
Because of the $S_5$ symmetry of the entire set-up, one only needs to compute
the action of a single $g_{i j}$ on $\mc{L}(X_{\mbb{1}})$. 
The ball quotient description of the moduli space from 
 \cite{act:cubic_surfaces} was only used to find the specific generators $g_{i j}$ 
 of the fundamental group $G$.
Our proofs of theorems \ref{theorem-monodromy-and-Petersen}, \ref{theorem-monodromy}.  
are completely elementary calculations,
and the ball quotient description play no role in the proof.
As expected, all of the geometry can be described in terms of the
combinatorics on the Petersen graph.
For a recent calculation of monodromy action of $G$ on a family of cubic surfaces branching 
over  smooth cubic curves, see \cite{Med}; in that case the image turns out
to be much smaller than $W(E_6)$.
\par
This paper is organized as follows.
Section  \ref{section-complex} is 
preparatory.
The $4$-ball quotient parametrizing $\mc{M}_{\msf{st}}$ and the ten generators
of the monodromy group $\Gamma$ is described in section \ref{section-ball}.
The Sylvester family and a neighborhood of the point $X_{\mathbb{1}}$ with 
the ten possible degeneration of the cubic form of $X_{\mathbb{1}}$ corresponding to the $10$ mirrors
of the meridians $g_{i j}$ are described in section \ref{section-Clebsch}.
After these preparations, the main computation of the action of $g_{i j}$ on the 
$\mc{L}(X_{\mathbb{1}})$ is carried out in section \ref{section-monodromy}.
%
%
%
\section{Preparation}
\label{section-complex}
Lemma \ref{l-monodromy-illustration}  and remark \ref{remark-Z-mod-2-monodromy}
 illustrates the idea of monodromy as used in this article. They are included only for 
 expository purposes and can be safely skipped. 
\begin{lemma}
\label{l-monodromy-illustration}
Let $w: [0, 1] \to \C$ be the parametrized unit circle  $w(t) = e^{2 \pi i t}$.
Consider the one parameter family of degree $n$ polynomials in a variable $x$ given by
$f_t(x) = x^n - w(t)$. 
\begin{enumerate}
\item There exists unique continuous function $a(t): [0, 1 ] \to \C$
satisfying the conditions that $a(t)$ is a solution of $f_t(x) = 0$ and that $a(0) = 1$.
\item Assume $n = 2$. There exists unique continuous functions $a(t), b(t): [0, 1 ] \to \C$
satisfying the conditions that $f_t(a(t)) = f_t(b(t)) = 0$ and $(a(0), b(0)) = (1, -1)$.
\end{enumerate}
\end{lemma}
\begin{proof}
The $n$ roots of $f_t(x) $ are $e^{2 \pi i (k + t) /n }$ where $k = 0, 1, \dotsb, n -1$ . 
The condition $a(0) = 1$ together with continuity
of $a(t)$ implies that $a(t)$ must be the solution branch $a(t) = e^{2 \pi i t/n}$. This proves part (a).
Part (b) is immediate from part (a) since $a(t) + b(t)  =0$ for all $t$.
\end{proof}
\begin{remark}
\label{remark-Z-mod-2-monodromy}
Let $\C^2 \simeq P_2 \subseteq \C[x]$ be the space of polynomials of the form $f = x^2 - s_1 x + s_0$.
Define $\tilde{P}_2  \subseteq  P_2 \times \C $ by $\tilde{P}_2 = \lbrace (f, c) \colon f(c) = 0 \rbrace$.
This defines the canonical ramified double cover 
\[
\tilde{\pi} : \tilde{P}^2 \to P_2
\]
 where ``the fiber over $f$ are the roots of $f$".
To realize this double cover another way, let $\pi: \C^2 \to P_2$ be the map 
$\pi : (a, b) \mapsto f_{a, b}(x) = (x - a)( x - b)$. Now the preimage of $f_{a, b}$ are $(a, b)$
and $(b, a)$.
We can identify these two double covers
of $P_2$ via the isomorphism $ \C^2 \to \tilde{P}_2$ defined by $(a, b)  \mapsto (f_{a, b}, a)$.
So we can identify 
\[
P_2 \simeq \C^2/S_2 = \lbrace \text{space of two points on $\C$} \rbrace
\]
 where $f_{a, b}$ is identified with the
``unordered pair" $(a, b) S_2$.
The preimage of  $f_{a, b}$ in $\C^2$  is $\lbrace (a, b), (b, a) \rbrace$
and the preimage of $f_{a, b}$ in  $\tilde{P}_2$ is 
$\lbrace (f_{a, b}, a), (f_{a, b}, b) \rbrace$. 
Let $\tilde{\Delta}$ be the diagonal in $\C^2$ and let $\Delta$ be its image in $P_2$. So
$\Delta = \lbrace (x - a)^2 \colon a \in \C \rbrace$. Write $P_2^{\circ} = P_2 - \Delta$.
\par
The fundamental group of $P^{\circ}_2$ is the two strand braid group $B_2 \simeq \Z$ generated by
a nontrivial loop $\mu$ that can be represented by the 
 family $f_t(x)= x^2 - w(t)$ based at $f_0$.
 The lemma above says that as we go around this loop $\mu$ 
  in $P_2$, upstairs in $\tilde{P}_2$, 
 the two points in the fiber can be uniquely continuously moved (in $\pi^{-1}(P_2^{\circ})$)
along the paths $a(t), b(t)$, so that when we
 come back to the base point $\lbrace 1, -1 \rbrace$ the two points in the fiber get exchanged, 
 that is, $a(1) = b(0) = -1$ and $b(1) = a(0) = 1$. This \textit{is} the monodromy
 action of the  braid group
 $B_2 = \langle \mu \rangle \simeq \Z$ on the fiber $\lbrace 1, -1 \rbrace = \pi^{-1}(f_0)$
 as the symmetric group $S_2 \simeq \Z/2$.
 Of course this discussion immediately generalizes to the symmetric group $S_n$ and the
 n-strand braid 
 group $B_n$.
 The example studied in this paper is more complicated, but similar in spirit. Indeed, the fundamental group of 
 the space of smooth cubic surfaces in $\mathbb{P}^3$ 
 studied below is ``braid-like" in the sense of \cite{AB1}. 
 The monodromy action of $g_A$'s in theorem \ref{theorem-monodromy-and-Petersen}
  was worked out while finding a braid-group like presentation for this fundamental group in \cite{ABL}.
 \end{remark}
 We end this section with a small lemma needed for proof of theorem \ref{theorem-monodromy}.
 \begin{lemma}
 \label{l-monodromy-of-roots-lemma}
Let  $f(z)$ be a holomorphic function with a zero of order $m \geq 1$ at some 
$c \in \C$.
There is a neighborhood $U$ of $c$ and a local analytic isomorphism 
$\varphi: U \to \varphi(U)$ such that $\varphi(c) = 0$ and $f(z) = \varphi(z)^m$.
Given a sufficiently small $\epsilon > 0$, the anti-clockwise loop 
$w(t) = \epsilon e^{ i t} $ for $0 \leq t \leq 2 \pi$,
and a solution  $z_0 \in U$ for the equation $f(z) = w(0)$, 
there is a unique continuous curve 
$z: [0, 2 \pi] \to U$ 
satisfying $z(0) = z_0$ and $f(z(t)) = w(t)$ for all $t$.
Explicitly, one has 
\begin{equation}
z(t) = \varphi^{-1}( e^{i t /m} \varphi(z_0 )) \text{\; for \;} t \in [0, 2 \pi].
\label{eq-def-z}
\end{equation}
In particular, $\varphi( z(2 \pi)) = e^{2 \pi i/m} \varphi(z_0)$. 
\end{lemma}
\begin{proof}
For $z$ in a small neighborhood $U$ of $c$, one can write 
$f(z) = (z - c)^m h(z)$ for some holomorphic $h$ that is defined on $U$
and does not vanish on $U$. 
Choose a holomorphic $m$-th root $g$ of $h$ and let 
$\varphi(z)=(z - c) g(z)$.
Then $\varphi(c)=0$ and $\varphi'(c)=g(c)\neq 0$. So, after possibly shrinking $U$, 
$\varphi : U \to \varphi(U)$ is a local analytic isomorphism and
$f(z)=(z-c)^m g(z)^m=\varphi(z)^m$ for all $z \in U$.
\par
Now choose $\epsilon>0$ sufficiently small to ensure that 
$ \lbrace u \in \C \colon \abs{u} = \epsilon^{1/m} \rbrace \subseteq \varphi(U)$.
Let $z_0 \in U$ be such that $f(z_0) = \epsilon$\footnote{Note that there are $m$ solutions of $f(z) = \epsilon$ inside $U$, namely
$\varphi^{-1}( \epsilon^{1/m} e^{2 \pi i r/m})$, with $r = 0, 1, \dotsb, m -1$.}.
Note that $\abs{\varphi(z_0)} = \epsilon^{1/m}$.
Define $z: [0, 2 \pi] \to U$ by equation \eqref{eq-def-z}.
This is a continuous curve in $U$ satisfying $z(0)=z_0$ and 
$f(z(t)) = w(t)$ for all $t \in [0, 2 \pi]$. This proves the existence of  $z(t)$.
\par
To prove uniqueness, let $z_1: [0, 2 \pi] \to U$ be another such curve. 
Then 
\[
\varphi( z_1(t))^m = f(z_1(t)) = w(t) = f(z(t)) =  \varphi( z(t))^m.
\] 
So $\varphi(z_1(t))/ \varphi(z(t))$ takes value in the $m$-th roots of unity, hence must be constant.
But $z_1(0) = z(0)$, so $\varphi(z_1(0)) = \varphi(z(0))$ and this implies 
$\varphi(z_1(t)) = \varphi(z(t))$ for all $t$. 
Since $\varphi$ is a local isomorphism on $U$, it follows that
$z_1(t) = z(t)$ for all $t$. 
\end{proof}
%
%
\section{The ball model and the monodromy group}
\label{section-ball}
\begin{definition}[A convenient model for $\mathcal{E}^{4,1}$ with visible $S_5$ symmetry]
Let $\mathcal{E}^{n,1}$ denote the rank $(n+1)$ free $\mathcal{E}$-module $\mathcal{E}^{n+1}$ with the
 hermitian form
\begin{equation*}
\ip{(x_0; x_1, \dotsb, x_n)}{(y_0; y_1, \dotsb, y_n)} = -x_0\bar{y}_0 + x_1 \bar{y}_1 + \dotsb + x_n \bar{y}_n.
\end{equation*}
Let $v_0 = (2; 1,1,1,1,1) \in \mathcal{E}^{5,1}$ and let 
\[
L= v_0^{\bot} = \lbrace (x_0; x_1, \dotsb,x_5) \in \mathcal{E}^{5,1} \colon x_1 + \dotsb + x_5 = 2 x_0 \rbrace.
\]
Since $v_0^2 =1$, one has $L \simeq \mathcal{E}^{4,1}$. 
Let $\mathbb{C} L := \mathbb{C} \otimes_{\mathcal{E}} L$ be the underlying complex vector space of $L$.
Let $\mathbb{P}: \mathbb{C} L \to \mathbb{P}( \mathbb{C} L)$ be the projection map. 
Recall that $\mathbb{B}(L)$ denotes the set of negative norm lines in the underlying complex vector space $\C L$
of $L$.
The ball $\mathbb{B}(L) \simeq \mathbb{B}^4_{\C}$ with the unique (up-to scale) $U(4,1)$ 
invariant metric is called the complex hyperbolic $4$-space.
The symmetric group $S_5$ acts on $L$ by permuting the coordinates $x_1, \dotsb, x_5$,
and hence acts on the ball $\mathbb{B}(L)$.
\end{definition}
\begin{definition}[roots and mirrors]
Vectors in $L(1) = \lbrace v \in L \colon v^2 = 1 \rbrace$ 
are called the {\it short roots} of $L$ since order six complex reflections 
in these vectors preserve $L$. Given a short root $s \in L(1)$, let 
$R_s  = R_s^{\zeta_6} \in \operatorname{Aut}( \mathbb{C} L)$ 
be the complex reflection that pointwise fixes the hyperplane 
$s^{\bot}$ and acts on $\mathbb{C} s$ as multiplication by
the sixth root of unity $\zeta_6 = e^{2 \pi i /6}$.
Let $\op{R}(L)$ be the subgroup of $\op{Aut}(L)$ generated by these order six complex reflections.
The group $\Gamma := \mathbb{P} R(L) \subseteq \mathbb{P} \! \op{U}(4,1)$ acts faithfully on $\mathbb{B}(L)$.
For each short root $s \in L(1)$, the totally geodesic hypersurface $\mathbb{B}(s^{\bot}) \simeq \mathbb{B}^3$
in $\mathbb{B}(L) \simeq \mathbb{B}^4$ is pointwise fixed by the complex reflections $R_s$; 
this hypersurface is called the {\it mirror} of $s$ (or of $R_s$).
Given $x, y, v \in \mathbb{C} L$ with $x^2 < 0$, $y^2 < 0$ and $v^2 > 0$,  one has:
\begin{equation*}
\cosh^2 d(x, y) = \frac{\ip{x}{y} \ip{y}{x}}{ \ip{x}{x} \ip{y}{y} }
 \text{\; and \;}
\sinh^2 d(x, v^{\bot}) =  -\frac{\ip{x}{v} \ip{v}{x}}{ \ip{x}{x} \ip{v}{v} }.
\end{equation*}
On the left hand side of these formulae, we have abbreviated $x$ instead of $\mathbb{P}(x)$ and
$v^{\bot}$ instead of $\mathbb{B}(v^{\bot})$ etc. We'll do this if there is no chance of
confusion.
\end{definition}
\begin{lemma}
\label{l-first-shell-around-tau}
 There is a unique point in $\mathbb{B}(L)$ fixed by $S_5$, represented by the vector
\begin{equation*}
\tau = (5; 2, 2, 2, 2, 2).
\end{equation*}
One has $\tau^2 = -5$. If $r = (r_0; r_1, \dotsb, r_5) \in L$, then $\ip{r}{\tau} = -r_0$.
The point $\tau$ is not on any short root mirror.
There are ten short root mirrors that are closest to $\tau$ and equidistant from it; namely the mirrors of 
$\alpha_{\lbrace i, j \rbrace} = e_0 + e_i + e_j$ where 
$e_0, e_1, \dotsb, e_5$ are the standard unit vectors in $\eu{E}^{5, 1}$ and
$\lbrace i, j \rbrace$ runs over the two element subsets of $\lbrace 1, 2, 3, 4, 5 \rbrace$.
\end{lemma}
\begin{proof}
We refer to the forthcoming preprint \cite{BM} since we do not need it for the proof of our main results.
\end{proof}
\begin{definition}[The period lattice from Petersen graph]
The ``period lattice" $L$ can be constructed from the Petersen graph $\eu{P}$ as follows.
Let $\mathcal{E} \eu{P}$ be the free hermitian $\mathcal{E}$-module with basis 
$\lbrace \alpha_{A}^{\circ} \colon A \in \eu{P}  \rbrace$ 
indexed by the vertices of the graph $\eu{P}$
with hermitian form satisfying
\begin{equation*}
\ip{\alpha_{A}^{\circ}}{\alpha_{B}^{\circ}} =  
\begin{cases} 
1 & \text{\; if } A = B \\
-1 &\text{\; if } A \cap B = \emptyset \\
0  & \text{\; otherwise}. 
\end{cases}
\end{equation*}
The lattice $\mathcal{E} \eu{P}$ is singular 
with a five dimensional radical modulo which it is an integral $\mathcal{E}$-lattice of signature $(4,1)$.
To describe this radical, 
for each $A \in \eu{P}$ let
\begin{equation*}
\tau_{A}^{\circ} = 2 \alpha_{A}^{\circ} + \sum_{B \colon A \cap B = \emptyset } \alpha_{B}^{\circ}.
\end{equation*}
Verify that 
$\ip{\alpha_{B}^{\circ}}{\tau_{A}^{\circ}} = -1$ for all $ A \in \eu{P}$.
So if $A, A'$ are two distinct vertices of $\eu{P}$, then
$(\tau_{A}^{\circ} - \tau_{A'}^{\circ})$ are orthogonal to each $\alpha_{B}^{\circ}$, hence
$\tau_{A}^{\circ} - \tau_{A'}^{\circ} \in \operatorname{radical}(\mathcal{E} \eu{P})$
and it is easy to see that the $\mathbb{Q}(\zeta_6)$-span of these vectors equals $\mathbb{Q}(\zeta_6) \otimes_{\mathcal{E}} \op{radical}(\mathcal{E} \eu{P})$.
\par
Let $ \alpha_{i j} = \alpha_{\lbrace i, j \rbrace}$ be the 
 be the $10$ roots in lemma \ref{l-first-shell-around-tau}
indexed by the vertices of the Petersen graph $\eu{P}$.
Note that the vectors $\lbrace \alpha_A \rbrace$ and $\lbrace \alpha_A^{\circ} \rbrace$
have the same inner products. It follows that
 the map $\mathcal{E}\eu{P} \to L$ defined by $\alpha_{A}^{\circ} \mapsto \alpha_{A}$
induces an isomorphism 
\begin{equation*}
\mathcal{E} \eu{P}/ \operatorname{radical}( \mathcal{E} \eu{P}) \simeq L \simeq \mathcal{E}^{4,1}.
\end{equation*}
The $10$ vectors $\lbrace \tau_A^{\circ} \colon A \in \eu{P} \rbrace$ in $\mathcal{E} \eu{P}$ determine 
a unique point in $L$; call it $\tau$ (this is the $\tau$ in lemma \ref{l-first-shell-around-tau}). 
So if $A$ is any vertex of $\eu{P}$, then 
\begin{equation*}
\tau = 2 \alpha_{A} + \sum_{B \colon A \cap B = \emptyset } \alpha_{B}.
\end{equation*}
\end{definition}
 \begin{definition}[The orbifold fundamental group]
Let $\op{Cox}(\eu{P}, n)$ denote the quotient of the Artin group of the graph $\eu{P}$ by
the relation that all ten Artin generators have order $n$.
Let $s, s' \in L(1)$ be two distinct short roots and let $R, R'$ be the $\zeta_6$-reflections in these.
One verifies that if $\ip{s}{s'} = 0$, then 
$R R' = R' R$ and if $\abs{\ip{s}{s'}} = 1$, then 
$R R' R = R' R R'$.
In particular, we have a map from $\op{Cox}(\eu{P}, 6)$
to $\Gamma$
obtained by sending the generators to the $\zeta_6$-reflections in 
$\lbrace \alpha_{i j} \colon \lbrace i, j \rbrace \in v(\eu{P}) \rbrace$.
From \cite{act:cubic_surfaces} we know that this map is onto. In fact seven of these reflections making an 
affine $E_6$-diagram generate $\Gamma$ and  $\Gamma = \operatorname{Aut}^+(L)$; see \cite{act:cubic_surfaces} 7.21.
Let $\eu{H}$ be the union of the mirrors of the short roots of $L$.
Let $\mathbb{B}(L)^{\circ} = \mathbb{B}(L) - \eu{H}$ denote the compliment of the mirrors in $\mathbb{B}(L)$.
Define the orbifold fundamental group 
\begin{equation*}
G =  \pi_1^{\operatorname{orb}}( \Gamma \backslash \mathbb{B}(L)^{\circ}, \tau).
\end{equation*}
For our purpose, an element of $G$ is a pair $(\gamma, \phi)$ where $\phi \in \Gamma$ and
$\gamma$ is a homotopy class of paths from $\tau$ to $\phi \tau$. 
The group law in $G$ is given by 
\[
(\gamma, \phi) (\gamma', \phi') = (\gamma * \phi \gamma', \phi \phi')
\]
where $\gamma * \phi \gamma'$ means $\gamma$ followed by  $\phi \gamma'$.
\end{definition}
\begin{definition}[Meridians in ball model]
\label{def-meridians-in-ball-model}
For each mirror $M_{i j} = \alpha_{i j}^{\bot}$ closet to $\tau$, there is an element
\begin{equation*}
(\mu_{i j}, R_{i j}) \in G
\end{equation*}
where $\mu_{i j}$ is a path based at  $\tau$
that goes around $\alpha_{i j}^{\bot}$ once. 
Here $R_{i j}$ denotes that $\zeta_6$-reflection in $\alpha_{i j}$. 
The path $\mu_{i j}$ consists of three segments, and is roughly described as follows:
Let $q_{i j}$ be the generic point on the mirror $M_{i j}$ that is closest to $\tau$
and let $\overline{ \tau, q_{i j}}$ be the geodesic ray joining $\tau$ to $q_{i j}$.
The path $\mu_{i j}$ follows $\overline{ \tau, q_{i j}}$
till it is very close to $q_{i j}$, then follows an counterclockwise arc going around
the mirror $M_{i j}$ making an angle $\pi/3$ and then follows the geodesic 
$\overline{q_{i j}, R_{i j} \tau}$ to $R_{i j} \tau$.
The element $(\mu_{i j}, R_{i j})$ or 
the path $\mu_{i j} $ representing it is called the \textit{meridian} (in the ball model) 
around  $M_{i j}$ based at $\tau$ (in the ball model);
for a careful definition of them, see  \cite{AB2}.
In the next section, we represent the meridians
$(\mu_{i j}, R_{i j})$ by explicit paths $X^{i j}_{u(t)}$ in the Sylvester family
of cubic forms.
\end{definition}
%
%
\section{A neighborhood of the Clebsch surface in Sylvester family}
\label{section-Clebsch}
\begin{definition}[The Sylvester family of cubic surfaces]
Let 
\[
\mathbb{P}^3_0 = \lbrace [x_1 \colon \dotsb \colon x_5  ] \in \mathbb{P}^4 \colon x_1 + \dotsb + x_5 = 0 \rbrace \simeq \mbb{P}^3.
\]
For each $\lambda \in \mathbb{P}^4$, define $X_{\lambda} \subseteq \mathbb{P}^3_0$ by
\[
\lambda_1 x_1^3 + \dotsb + \lambda_5 x_5^3 = 0.
\]
One knows that $X_{\lambda}$ and $X_{\lambda'}$ define the same cubic surface if and only if 
$\lambda$ and $\lambda'$ are in the same $S_5$ orbit.
The cubic surface represented by the form
$X_{\lambda}$ is denoted by $X_{\lbrace \lambda \rbrace}$ or simply by $X_{\lambda}$ if there is no 
chance of confusion.
Let $\Lambda \subseteq \mbb{P}^4$ be the set of parameters 
for which $X_{\lambda}$ defines a stable cubic surface.
The family 
\begin{equation*}
\mc{X}  = \lbrace X_{\lambda} \colon \lambda \in \Lambda \rbrace \to \Lambda.
\end{equation*}
is called the \textit{Sylvester family}.
\par
For $\lambda = \mbb{1} =  [1: 1: 1: 1: 1]$ we obtain the Clebsch surface 
$X_{\mbb{1}}$ with equations $x_1^3 + \dotsb + x_5^3 = x_1 + \dotsb + x_5 = 0$ having
 $S_5$-symmetry.
 \end{definition}
\begin{lemma}
There is a $S_5$-equivariant bijection $f$ 
from a neighborhood of $\tau$ in the ball $\mathbb{B}(L)$ to a neighborhood of the form $X_{\mathbb{1}}$
 in $\mc{X}$; in particular taking $f(\tau) = X_{\mathbb{1}}$ and 
 $f$ is compatible with the quotient maps to  $\mc{M}_{\msf{sm}}$.
 \label{lemma-identify-nbd-of-tau}
\end{lemma}

\begin{proof}[sketch of proof]
The Sylvester family yields a map
\[
\pi_{\mc{X}} : \mc{X} \to \mc{M}_{\msf{st}}
\]
that is an $S_5$-orbifold cover onto its image and it image
is known to be an open dense subset of $\mc{M}_{\msf{st}}$.
On the other hand the theorem of  \cite{act:cubic_surfaces}, quoted in \ref{theorem-ACT}, 
gives  us a $\Gamma$-orbifold cover 
\[
\pi : \mathbb{B}(L) \to \mc{M}_{\msf{st}}.
\]
It is known that $X_{ \mathbb{1} }$ is the only smooth cubic surface with $S_5$ symmetry.
Since $\tau \in \mathbb{B}(L)$ has stabilizer $S_5$ in $\Gamma$, it follows that $\pi(\tau)$ is the 
Clebsch cubic.
Furthermore, $\tau$ has an $S_5$-invariant neighborhood
$U \subseteq \mathbb{B}(L)^{\circ}$ such that $\pi$ restricts to an $S_5$-orbifold cover 
$\pi \vert_U: U \to \pi(U)$. By lifting property of orbifold covering spaces, after possibly shrinking $U$ 
to a smaller open set, we obtain an $S_5$-equivariant injective map $f: U \to \mc{X}$ 
 such that $\pi_{\mc{X}} \circ f = \pi$
 taking $\tau$ to the form $X_{\mathbb{1}}$.
 \end{proof}
\begin{topic}[Singular elements in $\mc{X}$]
The surface $X_{\lambda}$ is singular at some $x$ if and only if
$d (\sum_{i } \lambda_i x_i^3) = \sum_i 3\lambda_i x_i^2 dx_i$
is proportional to $d( \sum_i x_i) = \sum_i dx_i$, that is,
$ \lambda_i x_i^2 = \lambda_j x_j^2$ for all $i, j$, that is, 
$[ \dotsb: x_i \colon \dotsb] = [ \dotsb \colon \lambda_i^{-1/2} \colon \dotsb]$, so that $\sum_i x_i = 0$ implies
 $ \sum_i \lambda_i^{-1/2} = 0$.
\end{topic}
\begin{topic}[Deformations of Clebsch surface to singular surfaces]
We now describe ten ways to deform $X_{ \mbb{1} }$ to a singular cubic form with a
simple $A_1$ singularity along one parameter families in  $\lbrace X_{\lambda} \rbrace$.
For this, fix a two element subset $\lbrace i, j \rbrace \subseteq \lbrace 1, 2, 3, 4, 5 \rbrace$.
For $[s: t] \in \mathbb{P}^1$, let $X^{i j}_{s: t}$ denote the cubic form $X_{\lambda}$ where
$\lambda_i = \lambda_j = t$ and $\lambda_k = s$ if $k \in  \lbrace 1, 2, 3, 4, 5 \rbrace - \lbrace i , j \rbrace$.
For example 
\begin{equation*}
X^{4 5}_{s: t} = X_{ s: s: s: t : t }.
\end{equation*}
This defines ten one parameter families of cubic forms, parametrized by $  \mathbb{P}^1$.
The cubic surface $X_{ s: s: s: t : t }$
 is singular if and only if $2s^{-1/2} + 3 t^{-1/2} = 0$
that is, $s/4 = t/9$. So this one parameter family deforms
$X_{\mbb{1}}$ along a one parameter family to the singular surface
obtained when $(s, t) = (4, 9)$. 
\end{topic}
\begin{definition}[A one parameter subfamily of $\mc{X}$ containing a  meridian]
\label{def-meridian-in-Sylvester-family}
Let 
\[
X^{45}_w = X_{w} := X_{1:1:1:1 - 3 w: 1 - 3 w} \text{\; and \;}
\mc{X}^{45} := \lbrace \XFF_{w}  \colon w \in \C \rbrace.
\]
So $\mc{X}^{45}$ is a family 
parametrized by $\C$ (the specific choice of parameter $1 - 3 w$ is to simplify some calculation later on).
Note that $X_0 = X_{\mbb{1}}$ is the Clebsch surface.
The only singular surface in this family 
occurs when $w$ is equal to 
\[
\wstar = 5/27.
\]
The singular surface $\XFF_{\wstar} = X_{9:9:9:4:4}$ has just one singular point:
an ordinary double point at $[2: 2: 2:-3:-3]$.  
\footnote{This surface is given by the equation $f(x_2, x_3, x_4, x_5) = 
-9( x_2 + x_3 + x_4 + x_5)^3 + 9x_2^3 + 9x_3^3 + 4x_4^3 +  4x_5^3 = 0$
with singular point $[2:2:-3:-3] = [1:1:-3/2:-3/2]$.
On the affine patch $x_2 = 1$, it has the affine equation
$f(1, x_3, x_4, x_5)$
and  singular point $p = (1, -3/2, -3/2)$.
Take coordinate centered at $p$, i.e., define $(y_3, y_4, y_5) = (x_3, x_4, x_5) - p$.
Then in terms of $y_3, y_4, y_5$, the equation of the affine patch of $\XFF_{4/9}$ becomes
$f_2 + f_3 = 0$ where 
$f_2 = 9( 6y_3^2 + y_4^2 + y_5^2 + 6( y_3 y_4 + y_3 y_5 + y_4 y_5))$ 
and $f_3 = -9(y_3 + y_4 + y_5)^3 + 9 y_3^3 + 4 y_4^3 + 4 y_5^3 $.
Since the quadratic form $f_2$ is non-degenerate, this is a ordinary double point
(also called  a node or an $A_1$ singularity).}
\par
Fix a small positive real number $\epsilon$. Let $t_1 = 0$ and $t_2 = 2 \pi$.
Fix $t_0 < 0 < 2 \pi <  t_3$.
Consider a continuous path $w: [t_0 , t_3] \to \C$
that consists of three segments:
\begin{itemize}
\item A straight line $w^+(t)$ from $w = 0$ to $w = \wstar - \epsilon$ for $t \in [t_0, t_1]$,
\item followed by a counter-clockwise circle  $w^{\circ}(t) = \wstar - \epsilon e^{i t} $, for $t \in [t_1, t_2]$,
\item followed by a straight line $w^{-}(t)$ that is the reverse of $w^+(t)$ for $t \in  [t_2,  t_3]$.
\end{itemize}
Then $X^{45}_{w(t)} = X_{w(t)}$ determines a loop in the moduli space $\mc{M}_{\msf{sm}}$
starting and ending at $X_{\mathbb{1}}$.
Changing $\lbrace 4, 5 \rbrace$ with $\lbrace i, j \rbrace$ for each two element subset of $\lbrace 1, 2, 3, 4, 5 \rbrace$
we get ten explicit loops $X^{i j}_{w(t)}$ in $\mc{M}_{\msf{sm}}$.
These loops will be denoted by $g_{i j}$ and called meridians (in the Sylvester model).
 \end{definition}
 \begin{remark}
 Under the identification given by lemma \ref{lemma-identify-nbd-of-tau}, the first segment of the loop 
 $g_{i j}$ correspond to the first segment of the path $\mu_{i j}$ defined in 
 \ref{def-meridians-in-ball-model}.
 Thus, the loops $g_{i j}$ represent the elements
 $(\mu_{i j}, R_{i j}) $ defined in the orbifold fundamental group of the $4$-ball quotient 
 and the  ten mirrors closest to $\tau$ correspond to the 
ten possible degenerations of $X_{\mathbb{1}}$ to a surface with a single $A_1$-singularity.
 The claims made in this remark, i.e. the
 the equality of the meridians defined in the ball  and the meridians defined in the Sylvester family, 
can probably be proved using symmetry arguments in a manner 
 similar to theorem 5.2 of \cite{AB3}. We are going to skip these proofs since these
statements  are plausible from symmetry considerations and since the proof of our main theorem
\ref{theorem-monodromy-and-Petersen} (or  \ref{theorem-monodromy})
do not depend on these claims.
 In other words, the meridians in the ball quotient do not play a role in the actual proofs in 
 this article. We introduced them for two reasons. The first reason is that 
we got to the generators $g_{i j}$ while studying the ball-quotient description.
The second reason is that they illustrate the remarkable similarity of this example with the
thirteen dimensional ball quotient studied in the monstrous proposal \cite{AB1, AB2, AB3}. 
  \end{remark}
 %
%
  \section{Monodromy of the lines on Clebsch surface}
\label{section-monodromy}
 \begin{definition}[Lines on the cubic surfaces in the family $\mc{X}^{45}$]
 \label{def-family-45}
Let 
\[
\beta = \beta_- = ( 1 - \sqrt{5})/2 \text{\; and \;} \beta_+ = ( 1 + \sqrt{5})/2.
\]
Let
\[
e_1 = (1,0,0,0,0), \; e_2 = (0,1,0,0,0), \;
\dotsb, e_5 = (0,0,0,0,1), \text{\; and \;} e_{i j} = e_i - e_j.
\]
 For $x, y \dotsb \in \C^5$,
 we abbreviate 
 \[
 \mathbb{P} (x, y, \dotsb) := \mathbb{P} ( \C x + \C y + \dotsb).
 \]
   Take any $w \in \C$. Recall the family of surfaces $X_w = X_{1:1:1: 1 - 3w: 1 - 3w}$ given by 
   \[
   x_1 + x_2 + x_3 + x_4 + x_5 = x_1^3 + x_2^3 + x_3^3 + (1 - 3 w)( x_4^3 + x_5^3)= 0.
   \]   
  There are three immediately visible lines on $\XFF_w$ passing through the ``Eckardt point" 
  $\mathbb{P}(e_{45} )$,
 namely $\mathbb{P} ( e_{12}, e_{45})$,  $\mathbb{P} ( e_{13}, e_{45})$, $\mathbb{P} ( e_{23}, e_{45})$.
 The intersection of the plane $\mathbb{P} (e_{12}, e_{13}, e_{45} )$ and $\XFF_w$ is the union
 of these three lines. To find more lines on $\XFF_w$ for $w$ close to $0$,
 we consider the line 
 $L(0, 0) = \mathbb{P} ( e_{24}, e_{35})$ on $\XFF_0$ and
 we guess that when we move from $X_0$ to $X_w$ in the moduli space, the line 
 $L(0, 0)$ moves to a line on $X_w$ of the form
\[
  L(a, b) := \mathbb{P}( u_{a b} , v_{a b} )
\]
  where
 \[
   u_{ a b}  =  e_{24}  + a e_{23}  + b e_{21}     \text{\; and \;}
 v_{ a b}  =  e_{35} + a e_{32}  + b e_{31} .
  \]
 Lemma \ref{l-ab} below tells us when 
 a line $L(a, b)$ lies on $\XFF_w$ for $w$ close to $0$.
   \end{definition}
 %
 
  %
  \begin{lemma} 
  The line $L(a, b)$ lies on  $\XFF_{w}$ 
  if and only if 
  \[
  H(b) :=b^3 - b^2 - b = w \text{\; and  \;}
 g_2(a, b) :=(1 + b) a^2 + (1 + b)^2 a + b^3 = 0.
 \]
 \label{l-ab}
  \end{lemma}
\begin{proof}
Let 
\[
f_{w} (x_1, \dotsb, x_5) = x_1^3 + x_2^3 + x_3^3 + (1 - 3 w) (x_4^3 + x_5^3)
\]
 be the defining polynomial
of the hypersurface $\XFF_w$ in $\mathbb{P}^3_{0}$.
  The line $\mathbb{P}( u_{a b} , v_{a b} )$ lies on $\XFF_{w}$ if and only if
  $f_{w} ( s u_{a b} + t v_{a b}) = 0$  for all $s, t$.
  One computes
  \begin{equation*}
   f_{w} ( s u_{ a b} + t v_{ a b} ) 
  = 3 g_1(w, a, b) (s^3 + t^3)
  - 3 g_2( a, b) (s^2 t + s t^2)
  \end{equation*}
  where 
  \[
  g_1( w, a, b) = (a + b)( 1 +  a) (1 + b) + w = g_2(a, b) - (H(b) - w).
  \]
%
  So $\mathbb{P}( u_{a b} , v_{a b} )$ is a line on $\XFF_{w} $
if and only if  $g_1(w, a, b) = g_2(a, b) = 0$ if and only if
$H(b) - w = g_2(a, b) = 0$.
\end{proof}
\begin{example}
\label{example-of-lines-on-X1}
When $w = 0$, we find that $g_1(0, a , b) = 0$ if $a = -b$ and 
$H(b) = 0$ implies
$a = b = 0$ or $b^2 - b - 1 = 0$.  Which gives the solution 
$-a = b = \beta_{\pm} = (1 \pm \sqrt{5})/2$. This gives us the following two lines 
on $X_{\mbb{1}}$:
\begin{align*}
 L(-\beta_{\pm}, \beta_{\pm}) 
&= \mathbb{P}( e_{24}  - \beta_{\pm} e_{23}  + \beta_{\pm} e_{21}, e_{35} - \beta_{\pm} e_{32}  + \beta_{\pm} e_{31}) \\
&= \mathbb{P}( e_{24}  + \beta_{\pm} e_{31}, e_{35} + \beta_{\pm} e_{21})
\end{align*}
Permuting the coordinates, one obtains twelve distinct lines on  $X_{\mbb{1}}$ of this form.

\par
Below we show that for $ w \in [0, \wstar]$ 
the two equations in  \ref{l-ab} have a unique 
continuous family of real solutions $(a(w), b(w))$ satisfying
$(a(0), b(0)) = (0,0)$.
The singular cubic $\XFF_{\wstar}$ corresponds to $w= \wstar$ and 
one obtains 
\[
(a(\wstar),b(\wstar))= ((\tfrac{1}{6})^{1/2} - \tfrac{1}{3}, -\tfrac{1}{3}).
\]
This gives us some lines on the nodal cubic surface $\XFF_{\wstar}$.
\end{example}
Our objective now is to describe in detail  the set $\mc{L}(X_{\mbb{1}})$ of $27$ lines on $X_{\mbb{1}}$ 
 in  \ref{def-27lines}  and the action of meridians $g_{i j}$'s on 
 $\mc{L}(X_{\mbb{1}})$.
  Actually theorem \ref{theorem-monodromy} only describes the action of $g_{45}$. 
 The action of the other $g_{i j}$'s are obtained by using the $S_5$ symmetry.
The elements of $\mc{L}(X_{\mbb{1}})$
 are naturally parametrized by the lines and pentagons in the 
Petersen graph $\eu{P}$ (see figure \ref{figure-Petersen}). 
So first we describe the pentagons in $\eu{P}$.
\begin{definition}[Names for pentagons in Petersen graph]
\label{def-names}
The symmetric group $S_5$ acts on the two element subsets of $\lbrace 1, 2, 3, 4, 5 \rbrace$, hence
acts on the Petersen graph $\eu{P}$ as its full automorphism group.
The set  $\pentagon$ of Pentagons in $\eu{P}$
has size $12$ and is transitively permuted by $S_5$.
A convenient way to label these twelve pentagons is as follows.
Given distinct elements $a, b, c, d, e \in  \lbrace 1, 2, 3, 4, 5 \rbrace$, 
let $\pentagon_{a b c d e}$ be the pentagon  in $\eu{P}$ with vertex set
$\lbrace \lbrace a, b \rbrace, \lbrace b, c \rbrace, \lbrace c, d \rbrace, \lbrace d, e \rbrace, \lbrace e, a \rbrace \rbrace$.
Note that the pentagon depends on $a, b, c, d, e$ only up-to cyclic permutation of
elements and left-right reversal of the elements, i.e.
\[
\pentagon_{a b c d e} = \pentagon_{ b c d e a} = \dotsb
\text{\; and \;}
\pentagon_{a b c d e} = \pentagon_{ e d c b a}.
\]
So one can list the twelve elements of $\pentagon$ as 
$\pentagon_{a b c d 5}$ where $a, b, c, d \in \lbrace 1, 2, 3, 4 \rbrace$
are distinct elements and $a < d$.
\end{definition}
%
 %
%
\begin{definition}[Names for the $27$ lines on Clebsch surface]
\label{def-27lines}
The set  $\mc{L}(X_{\mbb{1}})$ of $27$ lines on $X_{\mbb{1}}$
naturally splits into two subsets $\mc{L}_{\edge}$ and $\mc{L}_{\pentagon}$ of size $15$ and $12$ respectively
that are naturally in bijection with the edges and pentagons in the Petersen graph $\eu{P}$ respectively.
Each pair of disjoint two element subsets
$E = \lbrace \lbrace \lbrace i, j \rbrace, \lbrace k, l \rbrace \rbrace$
of $\lbrace 1, 2, 3, 4, 5 \rbrace$ determine an edge of $\eu{P}$ and correspond to
 a line on $X_{\mbb{1}}$, namely,
\begin{equation}
L_{i j, k l} = 
L_E = \mbb{P} ( e_{ i j}, e_{k l} ).
\label{eq-def-L15}
\end{equation}
This names the fifteen lines of $\mc{L}_{\edge}$ using the edges of $\eu{P}$.
Let $\mc{L}_{\pentagon}$ be the remaining twelve lines.
We want to write down a bijection $\mc{L}_{\pentagon} \to \pentagon$.
We observed in \ref{example-of-lines-on-X1} that
$\mc{L}(X_{\mathbb{1}})$ contains elements of the from
$\mbb{P}( e_{i p} + \beta e_{j k} , e_{ j q} + \beta e_{i k} )$
where $\lbrace i, j, k, p, q\rbrace = \lbrace 1,  2, 3, 4, 5 \rbrace$.
One verifies that there are twelve lines on $X_{\mathbb{1}}$ of this form and that these
lines are all distinct from the lines in $\mc{L}_{\edge}$. 
So each element of $\mc{L}_{\pentagon}$ can be (non-uniquely) written in
this form.
Start with one of these lines, say $\mbb{P} ( e_{3 1} + \beta e_{5 4}  , e_{5 2} + \beta e_{3 4} )$.
Note that
\[
\mbb{P} ( e_{3 1} + \beta e_{5 4}  , e_{5 2} + \beta e_{3 4} )
 = \mathbb{P} \op{RowSpan}(A)
 \]
  where 
\[
A 
= [A_{\cdot, 1}, A_{\cdot, 2}, A_{\cdot, 3}, A_{\cdot ,4} , A_{\cdot, 5}]
=   \begin{bmatrix} 
1 & 0 &       -1 & \beta & -\beta \\
0 & 1 & -\beta & \beta & -1
 \end{bmatrix} 
=  -\begin{bmatrix} 
e_{3 1} + \beta e_{5 4}  \\
e_{5 2} + \beta e_{3 4} 
\end{bmatrix}.
\]
Given $\sigma \in S_5$, let 
\[
A^{\sigma} = 
[A_{\cdot, \sigma(1)}, A_{\cdot,  \sigma(2)}, A_{\cdot, \sigma(3)}, A_{\cdot ,\sigma(4)} , A_{\cdot, \sigma(5)}]
\]
 be the matrix obtained by permuting the columns of $A$  according to $\sigma$.
 This gives $5! = 120$ matrices $\lbrace A^{\sigma} \colon \sigma \in S_5 \rbrace$.
One verifies that
\[
\mbb{P} \op{RowSpan} A^{\sigma} \in \mc{L}_{\pentagon} \text{\; for all \;} \sigma \in S_5.
\]
Let $\gamma, \eta \in S_5$ be the elements $\gamma = (1 \; 2 \; 3 \;4 \; 5)$
and $\eta = (1 \; 4)(2 \; 3)$.
Verify that
\[
A^{\gamma} = \begin{bmatrix} 0 & -1 \\ 1 & -\beta \end{bmatrix} A
\text{\;\;\;\; and \;\;\;\;}
A^{\eta} = \begin{bmatrix} \beta & -1 \\ \beta  & -\beta \end{bmatrix} A
\]
The elements $\gamma$ and $\eta$ generate 
a dihedral group $ \langle \gamma, \eta \rangle$ of order $10$  and
\[
\op{RowSpan} A^{\sigma}   
=  \op{RowSpan} A \text{\; for all \;} \sigma \in \langle \gamma, \eta \rangle.
\]
This yields a bijection 
\begin{equation*}
\langle \gamma, \eta \rangle \backslash S_5 \to \mc{L}_{\pentagon}
\text{\; given by \;}
\langle \gamma, \eta \rangle \sigma \mapsto  \mathbb{P} \op{RowSpan} A^{\sigma}.
\end{equation*}
Now observe that the stabilizer of the pentagon 
$\pentagon_{1 2 3 4 5}$ in $S_5$ is also the dihedral group
$\langle \gamma, \eta \rangle$\footnote{The stabilizer of $\pentagon_{1 2 3 4 5}$
clearly contains $\langle \gamma, \eta \rangle$ so has order at least $10$.
Since $S_5$ acts transitively on the $12$ pentagons,  
the stabilizer of $\pentagon_{1 2 3 4 5}$ is exactly 
$\langle \gamma, \eta \rangle$.}.
This yields a bijection 
\[
\langle \gamma, \eta \rangle \backslash S_5 \to  \pentagon  
\text{\;\; given by \;\;}
\langle \gamma, \eta \rangle  \sigma \mapsto \sigma^{-1} \pentagon_{1 2 3 4 5}.
\]
Combining the two bijections above, we obtain a bijection
$\mc{L}_{\pentagon} \rightarrow \pentagon$ such that
\[
\mathbb{P} \op{RowSpan} A^\sigma \mapsto \sigma^{-1} \pentagon_{1 2 3 4 5}
\text{\; for all \;} 
\sigma \in S_5.
\]
Recall that we labeled the $12$ pentagons in $\eu{P}$ as $\pentagon_{a b c d 5}$
where $\lbrace a, b, c, d \rbrace = \lbrace 1, 2, 3, 4 \rbrace$ and $a < d$.
The line corresponding to $\pentagon_{a b c d 5}$ under the above bijection
will be denoted by $L_{ a b c d 5}$. 
Let $[a, b, c, d, e] \in S_5$ denote the permutation 
$\left( \begin{smallmatrix} 1 & 2 & 3 & 4 & 5 \\ a & b & c & d & e \end{smallmatrix} \right)$.
Since $\pentagon_{a b c d 5} = 
[a, b, c, d, 5]  \pentagon_{1 2 3 4 5}$, one has
 \[
 L_{ a b c d 5}
 = \mathbb{P} \op{RowSpan} A^ {[a, b, c, d, 5]^{-1} }.
 \]
 This completes naming of the $27$ lines on $X_{\mathbb{1}}$.
\end{definition}
\begin{theorem}[monodromy action on lines]
\label{theorem-monodromy}

Recall that $\beta_{\pm} = (1 \pm \sqrt{5})/2$ and $\beta = \beta_-$.
Let $i, j, k, p, q$ be indices such that $\lbrace i, j, k \rbrace = \lbrace 1, 2, 3 \rbrace$
and  $\lbrace p , q \rbrace = \lbrace 4, 5 \rbrace$.
\begin{enumerate}
\item[(1)] There are fifteen lines on $X_{\mbb{1}}$  fixed by $g_{45}$. These are:
\begin{enumerate}
\item three of the form $\mbb{P}(e_{i j}, e_{45})$, namely
$L_{1 2, 4 5}, L_{ 1 3, 4 5}, L_{2 3, 4 5}$.
\item six of the form $\mathbb{P}( e_{i j} , e_{k p})$, namely
$L_{12, 34}$, $L_{12, 35}$, $L_{13, 24},$ $L_{13, 25}$, $L_{23, 14}$, $L_{23, 15}$.
\item six 
of the form 
\(
\mathbb{P}( e_{i p} + \beta_+ e_{k j}, e_{k q} + \beta_+ e_{i j} ) 
= 
\mathbb{P}( e_{j i} + \beta e_{k q}, e_{k p} + \beta e_{j q}),
\)
namely, $L_{21435}, L_{12435}, L_{24135}, L_{13425}, L_{14235}, L_{14325}$.
\end{enumerate}
\item[(2)] There are six pairs lines on $X_{\mbb{1}}$ that are exchanged pairwise by $g_{45}$.
For each permutation $[i , j, k] = \bigl( \begin{smallmatrix} 1 & 2 & 3 \\ i & j & k \end{smallmatrix} \bigr) \in S_3$, one has a pair
\[
L_{i4, k5} = \mathbb{P}( e_{i 4} , e_{k 5} ), \;\;\;\;\;\;\;\;
L_{i j k 4 5} = \mathbb{P}( e_{i 4} + \beta e_{k j}, e_{k 5} + \beta e_{i j}  ).
\]
These twelve lines form a double-six configuration.
For each pair, $L_{i4, k5}$ and $L_{i j k 4 5}$
are on the opposite sides of the double-six.
 The three lines $\lbrace L_{i j k 4 5} \colon [i, j, k] \in A_3 \rbrace$
 are on one side of the double-six and the three lines
$\lbrace L_{i j k 4 5} \colon [i, j, k] \in S_3 - A_3 \rbrace$ are on the other side.
 \end{enumerate}
\end{theorem}
Before the proof, we extract Theorem \ref{theorem-monodromy-and-Petersen} from 
the statement of \ref{theorem-monodromy}.
\begin{proof}[proof of \ref{theorem-monodromy-and-Petersen}]
Consider figure \ref{fig-six-Petersens}. Let $A = \lbrace 4, 5 \rbrace$.
Each copy of $\eu{P}$ highlights an edge $E_i$ of the hexagon disjoint from $A$
and  highlights the pentagon $\pentagon(A, E_{i + 3})$ in the notation of Theorem 
\ref{theorem-monodromy-and-Petersen}.
Observe that these edges and pentagons are exactly the labels of the six pairs of lines of $\mc{L}(X_{\mbb{1}})$ 
in Theorem \ref{theorem-monodromy}, part (2).
By \ref{theorem-monodromy} these six pairs of lines are exchanged 
by the monodromy action of  $g_{45}$.
This verifies Theorem \ref{theorem-monodromy-and-Petersen} for $A = \lbrace 4, 5 \rbrace$.
Theorem \ref{theorem-monodromy-and-Petersen} then follows from $S_5$-symmetry.
\end{proof}

\begin{figure}

\newcommand{\PetersenHighlighted}[2]{%
\begin{tikzpicture}[
    scale= .5,
    transform shape,
    vertex/.style={
      circle,
      draw,
      fill=white,
      minimum size=2.2em,
      inner sep=0pt,
      font=\large
    },
    baseedge/.style={line width=.8pt},
    penthl/.style={draw=gray!65, line width=5pt, line cap=round},
    hexedgehl/.style={line width=2.4pt, line cap=round}
  ]

  \foreach \i in {1,...,5} {
    \coordinate (v\i) at ({90-(\i-1)*72}:3cm);
  }

  \foreach \i in {1,...,5} {
    \coordinate (u\i) at ({90-(\i-1)*72}:1.5cm);
  }

  \draw[penthl] #1;

  \draw[baseedge] (v1) -- (v2) -- (v3) -- (v4) -- (v5) -- (v1);

  \draw[baseedge] (u1) -- (u3) -- (u5) -- (u2) -- (u4) -- (u1);

  \foreach \i in {1,...,5} {
    \draw[baseedge] (v\i) -- (u\i);
  }

  \draw[hexedgehl] #2;

  \foreach \i/\label in {1/13, 2/24, 3/35, 4/41, 5/52} {
    \node[vertex] at (v\i) {\label};
  }
  \foreach \i/\label in {1/45, 2/51, 3/12, 4/23, 5/34} {
    \node[vertex] at (u\i) {\label};
  }

\end{tikzpicture}%
}

%
%
%

\newcommand{\PGraphOne}{%
  \PetersenHighlighted
    {(u1)--(v1)--(v5)--(u5)--(u3)--cycle}
    {(v2)--(v3)}
}

\newcommand{\PGraphTwo}{%
  \PetersenHighlighted
    {(u1)--(u3)--(u5)--(u2)--(u4)--cycle}
    {(v3)--(v4)}
}

\newcommand{\PGraphThree}{%
  \PetersenHighlighted
    {(u1)--(u4)--(u2)--(v2)--(v1)--cycle}
    {(v4)--(v5)}
}

\newcommand{\PGraphFour}{%
  \PetersenHighlighted
    {(u1)--(v1)--(v2)--(v3)--(u3)--cycle}
    {(v5)--(u5)}
}

\newcommand{\PGraphFive}{%
  \PetersenHighlighted
    {(u1)--(u3)--(v3)--(v4)--(u4)--cycle}
    {(u5)--(u2)}
}

\newcommand{\PGraphSix}{%
  \PetersenHighlighted
    {(u1)--(u4)--(v4)--(v5)--(v1)--cycle}
    {(u2)--(v2)}
}

\begin{tikzpicture}

  \begin{scope}[shift={(0, 2)}] %
    \node at (0,0) {\PGraphOne};
  \end{scope}

\begin{scope}[shift={(4, 2)}] %
    \node at (0,0) {\PGraphFive};
  \end{scope}

  \begin{scope}[shift={(8, 2)}] %
    \node at (0,0) {\PGraphThree};
  \end{scope}

  \begin{scope}[shift={(0,-2)}] %
    \node at (0,0) {\PGraphFour};
  \end{scope}

  \begin{scope}[shift={(4, -2)}] %
    \node at (0,0) {\PGraphTwo};
  \end{scope}

  \begin{scope}[shift={(8, -2)}] %
    \node at (0,0) {\PGraphSix};
  \end{scope}

\end{tikzpicture}
\caption{Labels for the six pair of lines on $X_{\mathbb{1}}$ that are exchanged by the action of
$g_{45}$ are highlighted to illustrate Theorem \ref{theorem-monodromy-and-Petersen}.}
\label{fig-six-Petersens}
\end{figure}
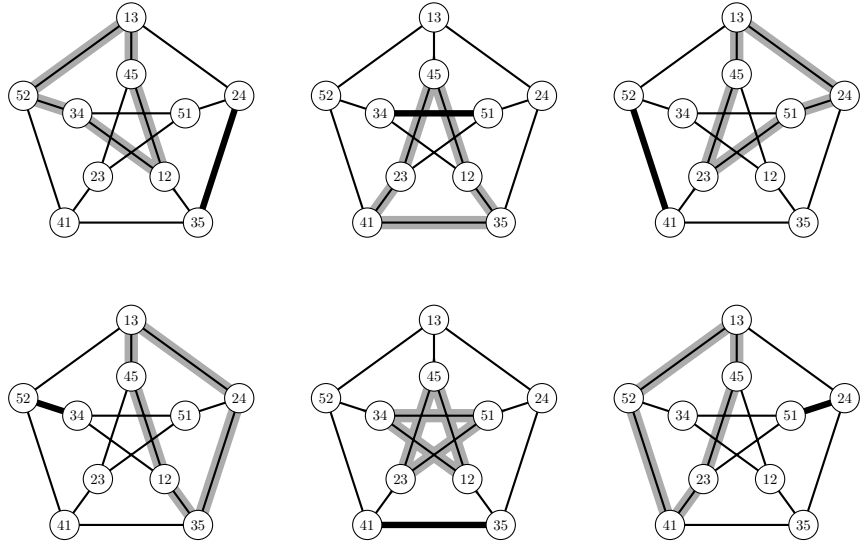

\begin{proof}[proof]
One verifies that the $27$ lines listed are distinct. So these are all the lines on $X_{\mbb{1}}$.
First we prove the more interesting part (2) and then part (1).
\par
{\it proof of (2)}. The argument for each pair is similar, so 
we verify that the pair
$\mathbb{P}( e_{2 4} , e_{3 5} )$ and $\mathbb{P}( e_{2 4} + \beta e_{3 1}, e_{3 5} + \beta e_{2 1}  )$
are exchanged by $g_{45}$. We discussed this pair 
in \ref{def-family-45}, \ref{l-ab}, \ref{example-of-lines-on-X1}
and we use the notation introduced there.
From \ref{def-meridian-in-Sylvester-family}, 
recall the path $w(t)$ and the explicit parametrization $X_{w(t)}$ of the meridian $g_{4 5}$. 
Recall that $w(t)$ depends on choice of a small positive real number $\epsilon$.
Recall from lemma \ref{l-ab} that the line 
\[
  L(a, b) = \mathbb{P}(  e_{24}  + a e_{23}  + b e_{21} , e_{35} + a e_{32}  + b e_{31} )
\]
  lies on $X_{w(t)}$ if and only if
\begin{equation}
H(b) - w(t) = g_2(a, b) = 0. 
\label{eq-Hg2}
\end{equation}
At the beginning of the path $X_{w(t)}$
 we start with the line
$L(0, 0) = \mathbb{P}( e_{2 4} , e_{3 5} )$ on $X_{\mathbb{1}}$, i.e.,
$(a, b) = (0, 0)$ when $w = 0$.
\par	 
We want to follow the solutions of 
$H(b) - w = g_2(a, b) = 0$
as $w$ varies along $w(t)$, starting at $(a, b) = (0,0)$ when $w = 0$.
We claim that there exists a unique continuous path of solutions $(a(w(t)), b(w(t)))$ of 
equation \eqref{eq-Hg2}
starting at  $(a, b) = (0, 0)$ and ending at $(a, b) = (-\beta, \beta)$.
The curves $w(t)$ and $b(w(t))$ are shown in figure \ref{fig:1}.
%
%
\begin{figure}
\begin{tikzpicture}
\begin{scope}[xshift = 0cm, xscale = 3.2, yscale = 3.2]
\draw [fill = white] (-1.43, -1.2) -- (2.43, -1.2) -- (2.43, .9) -- (-1.43, .9) -- (-1.43, -1.2); 
\draw[fill = lightgray, color = lightgray] (-1.4, -.4) -- (1.6,-.4) -- (2.4, .4) -- (-.6, .4) -- (-1.4, -.4);
\draw [arrows = {-Stealth[fill = none, inset=2pt, length=7pt, angle'=60, slant = .8]} ] (-1,0)--(2,0);
\draw [arrows = {-Stealth[fill = none, inset=2pt, length=7pt, angle'=60]} ] (0,-1)--(0,.6);
\draw [arrows = {-Stealth[fill = none, inset=2pt, length=6pt, angle'=60, slant = -.8]} ] (0, 0)--(.4,.4);
\draw [red, line width = .6pt] (1.618, 0)--(1.66,0);
\draw [purple, line width = .6pt] (1.68, 0) to [out=60, in = 60] (1.66,0);
\draw[ultra thin, olive] (5/3, -1) -- (5/3, .8); 
\draw [purple, line width = .6pt] (1.66, 0) to [out=240, in = 240] (1.68,0);
\draw [red, line width = .6pt] (0,0)--(-1/3 +.08,0);
\draw [red, line width = .6pt, arrows = {-Stealth[fill = none, inset = 0pt, length = 4pt]}] (0,0) -- (-.12,0);
\draw [red, line width = .6pt] (-1/3 -.08,0) -- (-.618,0);
\draw [red, line width = .6pt, arrows = {-Stealth[fill = none, inset = 1pt, length = 4pt]}] (-1/3-.08,0) -- (-1/3-.2,0);
\draw [purple, line width = .6pt] (-1/3 +.08, 0) to [out=60, in=60] (-1/3 - .08,0);
\draw [fill] (-1/3+.08, 0) circle (.004cm); 
\draw [fill] (-1/3-.08, 0) circle (.004cm); 
\draw[ultra thin, olive] (-1/3, -1) -- (-1/3, .8); 
\draw[scale=1, domain=-1:1.8, smooth, variable=\x, blue, line width=.24] plot ({\x}, {\x*\x*\x - \x*\x  - \x});
\draw[ultra thin, olive] (-1, 5/27) -- (2, 5/27); 
\draw[purple, ultra thin] (-1, 5/27-.012) -- (2, 5/27 - .012); 
\draw[purple, ultra thin] (-1/3+ .08, -1) -- (-1/3 +.08, .5); 
\draw[purple, ultra thin] (-1/3- .08, -1) -- (-1/3 -.08, .5); 
\draw[purple, ultra thin] (5/3- .004, -1) -- (5/3 -.004, .5); 
\draw[ultra thin, arrows = {-Stealth[fill = none, inset=.5pt, length=1.5pt, angle'=60]} ] (-1/3+.04, -1)--(-1/3+.07,-1);
\draw[ultra thin, arrows = {-Stealth[fill = none, inset=.5pt, length=1.5pt, angle'=60]} ] (-1/3+.04, -1)--(-1/3+.01,-1);
\node at ( -1/3 +.05 ,-1.05) {\tiny $\epsilon_1$};
\draw[ultra thin, arrows = {-Stealth[fill = none, inset=.5pt, length=1.5pt, angle'=60]} ] (-1/3-.04, -1)--(-1/3-.07,-1);
\draw[ultra thin, arrows = {-Stealth[fill = none, inset=.5pt, length=1.5pt, angle'=60]} ] (-1/3-.04, -1)--(-1/3-.01,-1);
\node at ( -1/3 -.05 ,-1.05) {\tiny $\epsilon_2$};
\node at (-.9, 5/27 - .05) {\tiny \textcolor{purple}{ $w = 5/27 - \epsilon$}};
\node at (-.95, 5/27 + .04) {\tiny  \textcolor{olive}{$w = 5/27$}};
\node at (-.63, .035) {\tiny $\beta$};
\node at (1.59, .035) {\tiny $\beta_+$};
\draw [fill] (0,0) circle (.004cm); 
\draw [fill] (-1/3,0) circle (.004cm); 
\draw [fill] (-.618,0) circle (.004cm);
\draw [fill] (1.618,0) circle (.004cm);
\draw [fill] (0,5/27) circle (.004cm);
\node[rotate=90] at (-1/3-.03,.7) {\tiny \textcolor{olive}{$x = -1/3$}};
\node[rotate=90] at (5/3-.03,.7) {\tiny \textcolor{olive}{$x = 5/3$}};
\node[xslant=.5, rotate = -20] at (-1.3,-.35) {\tiny $b$};
\node[xslant =.5, rotate = 0] at (.44,.45) {\tiny $Im(b)$};
\node[xslant =.5, rotate = 0] at (2.2,0) {\tiny $x = Re(b)$};
\node at (.06,.6) {\tiny $w$};
\node [rotate = 80] at (1.84,.62) {\tiny \textcolor{blue}{$w = H(b) = b^3 - b^2 - b$}};
\node at (-.12, -.05) {\tiny \textcolor{red}{$b(w(t))$}};
\end{scope}
\begin{scope}[xshift=3.7cm, yshift = 2.4cm, xscale = 3.2, yscale = 3.4]
\draw [fill = white] (-.3,-.3)--(.3,-.3)--(.3,.3)--(-.3,.3)--(-.3,-.3);
\node at (-.25,-.25) {\tiny $w$};
\draw [red, line width = .6pt] (0,0)--(5/27 - .02, 0);
\node at (0,-.04) {\tiny $0$};
\node at (.1, .05) {\tiny \textcolor{red}{$w(t)$}};
\node at (5/27,-.07) {\tiny $5/27$};
\draw [red, line width = .6pt, arrows = {-Stealth[fill = none, inset = 0pt, length = 4pt]}] (0,0) -- (.07,0);
\draw [red, line width = .6pt, arrows = {-Stealth[fill = none, inset = 1pt, length = 4pt]}] (5/27-.02,0) -- (5/27-.09,0);
\draw [purple, line width =.6pt] (5/27,0) circle (.02cm);
\draw [fill] (0,0) circle (.004cm); 
\draw [fill] (5/27-.02,0) circle (.004cm);
\draw [fill] (5/27,0) circle (.004cm);
\end{scope}
\begin{scope}[xshift=1cm, yshift = 2.82cm, xscale = 4, yscale = 4]
\draw [fill = white] (-.4,-.15)--(.4,-.15)--(.4,.15)--(-.4,.15)--(-.4,-.15);
\draw [red, line width = .6pt] (-.2,0) --(0,0);
\draw [red, line width = .6pt, arrows = {-Stealth[fill = none, inset = 0pt, length = 4pt]}] (-.2,0) -- (-.12,0);
\draw [red, line width = .6pt] (.1,0) --(.3,0);
\draw [red, line width = .6pt, arrows = {-Stealth[fill = none, inset = 1pt, length = 4pt]}] (.1,0) -- (.18,0);
\draw [purple, line width = .6pt] (0,0) --(.1,0);
\node at (-.2, .04) {\tiny $t_0$};
\node at (0,.04) {\tiny $t_1$};
\node at (0,-.04) {\tiny $0$};
\node at (.12, .04) {\tiny $t_2$};
\node at (.12,-.04) {\tiny $2 \pi$};
\node at ( .33, .04) {\tiny $t_3$};
\node at (-.35, -.1) {\tiny $t$};
\draw [fill] (-0.2,0) circle (.004cm);
\draw [fill] (0,0) circle (.004cm);
\draw [fill] (.1,0) circle (.004cm);
\draw [fill] (.3,0) circle (.004cm);
\end{scope}
\end{tikzpicture}
\caption{The curve $w(t)$ is shown in the $w$-plane (inset) and the curve
  $b(w(t))$ is shown in the gray shaded $b$-plane. The red segments of 
  $w(t)$ are $w^{\pm}(t)$ and the purple circular segment is $w^{\circ}(t)$.
  Similarly for $b$. The horizontal thin brown line is $w = \wstar = 5/27$ and
  the horizontal thin purple line is $w =\wstar - \epsilon$}
   \label{fig:1}
\end{figure}
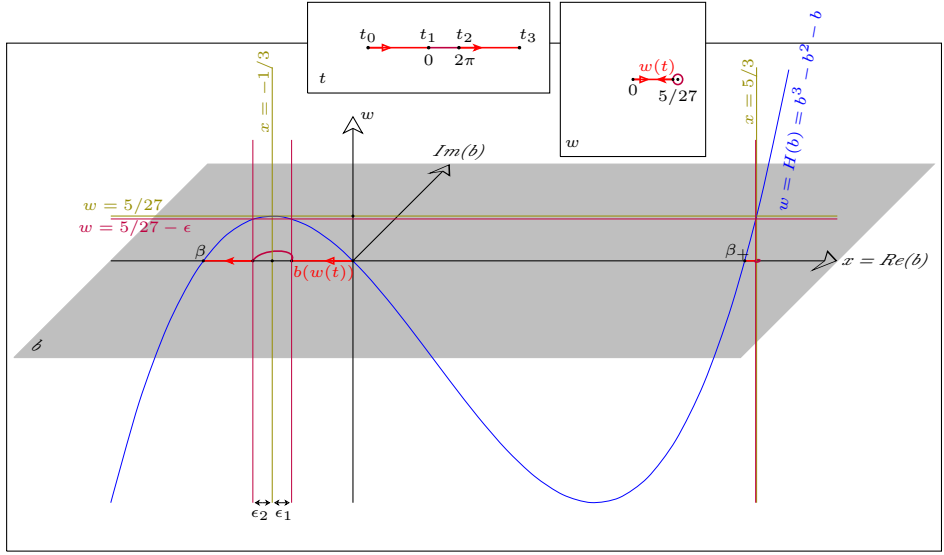
%
%
The claim implies, as 
we move in the moduli space along the loop $g_{45}$, there is a unique continuous way to move
the line $L(0,0) = \mathbb{P} (e_{24}, e_{35} )$ along $L(a(w(t)), b(w(t))) \subseteq \XFF_{w(t)} $ 
so that when we come back to $X_{\mbb{1}}$, the line $L(0,0)$ moves to 
\[
L(-\beta, \beta)
= \mathbb{P}( e_{24}  + \beta e_{31}, e_{35} + \beta e_{21}).
\] 
%
\par
It remains to prove the claim.
Consider the polynomial $H(b) = b^3 - b^2 - b$. It has three real roots $\beta = \beta_-$, $0$, $\beta_{+}$,
 is nonnegative for $[\beta_-, 0]$ and $[\beta_+, \infty)$, is negative otherwise,
and it has a unique local maximum at
\[
b_* = -1/3 \text{\; with \;} 
H(\bstar) = \wstar = 5/27.
\]
Using intermediate value theorem, 
it follows that all three roots of $H(b) =  w $ are real for $w \in [0,\wstar]$
with a double root $b = \bstar$ for $w = \wstar$. Recall that $w^+(t)$ ends
and $w^-(t)$ begins at $\wstar - \epsilon$. Let $b = \bstar + \epsilon_1$ and 
$b = \bstar - \epsilon_2$ be the two roots of $H(b) =  \wstar - \epsilon$
near $\bstar$ where $\epsilon_1, \epsilon_2$ are small positive real numbers; see figure \ref{fig:1}.
From the inverse function theorem of one variable calculus, 
 it follows that as $w$ moves along $w^+(t)$ from $0$ to $\wstar - \epsilon$, 
 the equation $H(b)  = w$ has a unique
continuous family of solutions $b^+(w)$ starting at $0$
and this solution curve $b^+(w)$ monotonically decreases from 
 $0$ to $\bstar + \epsilon_1$.
 \par
%
Next comes the key part of the calculation where
we need to follow the root of $H(b) = w$ along the circle $w^{\circ}(t)$.
For this, we let 
\[
z = b - \bstar, \text{\; so that \;}
H(b) - w = (\wstar - w) - z^2 ( 2 - z).
\]
As $w$ moves along $w^{\circ}(t)$, 
the variable $(\wstar - w)$ makes one anti-clockwise rotation along the circle
$\wstar - w^{\circ}(t) = \epsilon e^{i t}$ for $0 \leq t \leq 2 \pi$. 
The holomorphic map 
\[
z \mapsto z^2 ( 2 - z)
\]
 is a ramified double cover
near the origin, so there is a unique continuous branch of solution
$z^{\circ}(t)$ of 
\[
z^2 ( 2 - z) = \epsilon e^{ i t}
\text{\; for \;}0 \leq t \leq 2 \pi
\]
starting at $z^{\circ}(0) = \epsilon_1$ and one has
$z^{\circ}( 2 \pi) = -\epsilon_2$ 
(see \ref{l-monodromy-of-roots-lemma})\footnote{In the
notation of \ref{l-monodromy-of-roots-lemma}, one has the
local analytic isomorphism $\varphi(z) = z \sqrt{2 - z}$  near $0$.
So $z^{\circ}(t) = \varphi^{-1}( \sqrt{\epsilon} e^{ i t/2} )$.
By definition, $\epsilon_1, -\epsilon_2$ are the two real solutions of 
$ z^2( 2 - z) = \epsilon$, i.e. $\varphi(z)^2 = \epsilon$ near $z = 0$.
Since the real function $z \mapsto \varphi(z)$ is monotone increasing near $0$, 
one has
$\epsilon_1 = \varphi^{-1}( \sqrt{\epsilon})$ and $-\epsilon_2 = \varphi^{-1}( -\sqrt{\epsilon})$.
It follows that $z^{\circ} (2 \pi) = \varphi^{-1}( -\sqrt{\epsilon}) = -\epsilon_2$. 
}.
Thus we find that  as $w$ varies along $w^{\circ}(t)$, the equation
 $H(b) = w$ has a unique continuous family of complex solutions
 $b^{\circ}(t) = \bstar +  z^{\circ}(t)$ that starts at $\bstar + \epsilon_1$ and ``rotates
 by $180$ degrees anti-clockwise in the $b$-plane around $b_*$" to end at $\bstar - \epsilon_2$.
 \par
Finally, again using the intermediate value theorem, it follows that
as $w$ moves along $w^-(t)$ from $\wstar - \epsilon$ to $0$, 
 the equation $H(b) = w$ has a unique
continuous family of solutions $b^-(w)$ starting at $\bstar - \epsilon_2$
and this solution curve $b^-(w)$ ends at $\beta = \beta_-$.
\par
Thus we find that 
$H(b ) = w(t)$ have a unique continuous family of solutions $b(w(t))$ given by 
$b^+(w^+(t))$ followed by $b^{\circ}(w^{\circ}(t))$ followed by $b^-(w^-(t))$ 
as shown in figure \ref{fig:1}.
The solutions $b(w)$
are real and lies in the range $[\beta_-, 0]$ except 
when $w$ is in the circular part $w^{\circ}(t)$.
Now from the equation $g_2(a, b) = 0$ (which is quadratic in $a$)
it is easy to see that
\[
a(w) = - \tfrac{b(w) + 1}{2}  +  \sqrt{    \left( \tfrac{ b(w) + 1}{2} \right )^2 -  \tfrac{b(w)^3}{ b(w) + 1} }
\]
is the unique continuous branch of solutions for $a$ starting at $a(0) = 0$.
One verifies that the quantity under the square root sign takes the value $1/6 \neq 0$,
when $b(w) = b_* = -1/3$, so
there is no problem taking complex square roots in the portion $b^{\circ}$ since
we are away from the branch point.
This shows the existence of the unique path of solutions $(a(w(t)), b(w(t)) )$ to equation
\eqref{eq-Hg2}
 starting at $(0,0)$.
 One verifies that except for $w$ in the circular part $w^{\circ}(t)$
the solutions $(a(w), b(w))$ are real, with $b(w(t))$ monotonically decreasing and
consequently $a(w(t))$ monotonically increasing with $t$.
 In particular, $a(w) $ is real and positive except for $w$ in the circular part $w^{\circ}(t)$.
 At the end of the path, i.e. for $t = t_3$, we have $w = 0$, so
 $g_1(w, a, b) = 0$ (with $g_1$ as in the proof of \ref{l-ab})
 implies $a + b = 0$ (since positivity of $a$ rules out $a = -1$). Since $b = \beta$, we have $a = -\beta$.
This proves the claim and shows that the monodromy action of $g_{4 5}$
 moves $L(0, 0)$ to $L(-\beta, \beta)$.  
 Finally, applying the same argument starting from $L(-\beta,\beta)$ shows that
 $g_{4 5}$ moves $L(-\beta, \beta)$ back to $L(0, 0)$.
 This proves part (2).
\par
{\it proof of (1a)}. The three lines in (1a) are fixed lines on each $\XFF_w$ and thus are obviously fixed by 
$g_{45}$.
\par
{\it proof of (1b)}. The argument for each line in part (b) is similar. Consider the line $\mathbb{P}( e_{12}, e_{34})$
and its one parameter deformation of the form 
\[
L(b) = \mathbb{P}( e_{12} , e_{34} + b e_{35})
\]
 (with $b \neq -1$; because
$b = -1$ gives back the fixed line $\mathbb{P}( e_{12}, e_{45})$).
 The line $L(b)$ lies on
$\XFF_{w}$ if and only if\footnote{
With $f_w$ as in proof of
\ref{l-ab}, verify that $f_w( s e_{12} + t (e_{34} + b e_{35}) )= 3 t^3 ( 1 + b) h_1(w, b)$.} 
\begin{equation*}
h_1(w, b) := w b^2 + (1 -w) b + w = 0,
\end{equation*}
that is,
\begin{equation*}
b = b_{\pm}(w) := \tfrac{ w - 1  \pm \sqrt{ (1 - w)^2 - 4 w^2} }{2 w} \text{\; for \;} w \neq 0; 
\text{\; and \;}
b = b_+(0) := 0 \text{\; for \;} w = 0.
\end{equation*}
We want to follow the solutions of $h_1(w, b) = 0$ 
as $w$ varies along $w(t)$ starting at $b = 0$ when $w = 0$.
We claim that there exists a unique continuous path of solutions $b(w(t))$ of $h_1(w(t), b) = 0$
starting at  $0$ and ending at $0$, namely
$t \mapsto b_+(w(t))$. 
First note that as $w \to 0$, the expression $b_-(w)$ does not have a limit 
while $b_+(w)  \to 0$.
 So as we start moving along $w^+(t)$, there is a unique continuous 
path of solutions of $h_1(w, b) = 0$ starting at $b(w(t_0)) = b(0) = 0$, namely $b_+(w(t))$. 
Now to verify the claim, note that
as $w$ increases from $0$ to $\wstar$,  the quantity $ h_2(w) := (1-w)^2 - 4 w^2$ 
monotonically decreases from $h_2(0) = 1$ to $h_2(\wstar) = 2^7 3^{-5}$. 
So near $w = \wstar$ in the complex plane, $h_2(w)$ has a well
defined single-valued holomorphic square root that extends the real square root
taken along $w^+(t)$.
It follows that  $b_+(w^{\circ} (t))$ is a closed loop.
Finally along the reverse segment $w^{-}(t)$, the solution path $b_+(w^-(t))$ 
retraces the path $b_+(w^+(t))$ in reverse back to $b_+(w(t_0)) = b_+(0) = 0$.
\par
{\it proof of (1c)}. 
The six lines listed in part (1c) are of the form
\[
\mathbb{P}( e_{i p} + \beta_+ e_{k j}, e_{k q} + \beta_+ e_{i j} ) 
\]
The argument is the same for each of these.
So pick the line
\[
\mathbb{P} ( e_{24} + \beta_+ e_{31}, e_{35} + \beta_+ e_{21} ) = L(-\beta_+, \beta_+),
\]
which belongs to the 2-parameter family $L(a, b)$ described in
\ref{def-family-45} and argue that it is fixed by the monodromy action of $g_{45}$.
The details are similar to the proof of part (b) given above, so we'll be brief.
We need to follow the solution of \eqref{eq-Hg2} along the path $w(t)$
 starting at $(a, b) = (-\beta_+,\beta_+)$ when $w = 0$ (see \ref{l-ab}). 
We claim that there exists a unique continuous path of solutions 
$(a(w(t)), b(w(t)))$ of 
equation \eqref{eq-Hg2}
along $w(t)$
starting and ending at  $(-\beta_+, \beta_+)$. 
Note that 
\[
H(b) - \wstar  = (b + \tfrac{1}{3} )^2 ( b - \tfrac{5}{3}).
\]
Since $H$ is strictly increasing on $[\beta_+, 5/3]$,
using the inverse function theorem,  it follows that there is a unique 
continuous function $b^+(w^+(t))$ such that $H( b^+(w^+(t)) ) = w^+(t)$ satisfying
 $b^+(0) = \beta_+$ and $b^+$ monotonically increases 
from $\beta_+$ to $(5/3 - \epsilon_3)$ (for some $\epsilon_3 > 0$) 
as $w^+(t)$ goes from $0$ to $(\wstar - \epsilon)$.
The argument for the circular part is similar to the argument in part (1).
Setting $b_2 = 5/3 - b$, note that 
\[
H(b) - w = (\wstar - w) - (2 - b_2)^2  b_2,
\]
so we have to find a continuous family of solutions of the equation
\[
(2 - b_2)^2 b_2 = \epsilon e^{i t} \text{\; for \;} 0 \leq t \leq 2 \pi
\]
 starting at $b_2 =  \epsilon_3$
when $t = 0$. Since the complex analytic function
\[
b_2 \mapsto (2 - b_2)^2 b_2
\]
 is a local isomorphism
near $0$, there is a unique closed loop $b_2^{\circ}(w^{\circ}(t))$ solving this equation
starting with $b_2^{\circ} = \epsilon_3$ when $t = 0$. Setting $b^{\circ} = 5/3 - b_2^{\circ}$, we obtain
the unique closed loop $b^{\circ}(w^{\circ}(t))$ satisfying
$H( b^{\circ}(w^{\circ}(t)) ) = w^{\circ}(t)$, starting and ending at $b^{\circ}(0) = 5/3 - \epsilon_3$.
Finally, as we go back along $w^{-}(t)$ there is a unique continuous path of 
solutions to $H(b) = w$ starting at $5/3 - \epsilon_3$ and this solution just follows the
reverse of $b^+$. It follows that there is 
a unique continuous path of solutions 
$b(w(t))$ of $H(b) = w(t)$ starting and ending at  $\beta_+$. 
Now 
we can uniquely continue the branch of solutions for $a$ along $w(t)$ starting from $a = -\beta_+$ 
using the quadratic equation $g_2(a, b) = 0$ as in the proof of part (2)
(making sure that there is no problem taking the square root when $b(w(t))$ goes through the circular part)
and get a unique path of solutions $a(w(t))$ ending at $-\beta_{+}$.
This verifies part (1c).
\end{proof}
\begin{remark}
Let  $\rho:G  \to \op{Perm}( \mc{L}(X_{\mbb{1}}) )$ be the monodromy action
of the fundamental group $G = \pi_1( \mc{M}_{\msf{sm}}, X_{\mbb{1}} )$ on $ \mc{L}(X_{\mbb{1}} )$.
The theorem \ref{theorem-monodromy} above describes the permutation $\rho( g_{45})$.
Because of the $S_5$ symmetry, 
the action of all ten generators $g_{i j}$ are immediately
obtained by permuting the coordinates. One verifies
that the ten permutations $\lbrace \rho(g_{A}) \colon A \in  \eu{P}  \rbrace$
satisfies the Coxeter relations of the Petersen Graph $\eu{P}$ and the ``deflation relation" for each hexagon in $\eu{P}$.
If 
$A_1, \dotsb, A_6$ are the vertices of a hexagon in $\eu{P}$ labeled in cyclic order, 
and if we write $\rho_i = \rho(g_{A_i})$, 
then the deflation relation for this hexagon means
\[
\rho_1 \rho_2 \rho_3 \rho_4 \rho_5 = \rho_2 \rho_3 \rho_4 \rho_5 \rho_6.
\]
This is the relation that collapses the affine Weyl group of type $\tilde{A}_5$ to 
the spherical Weyl group of type $A_5$.
From this, after some calculation, 
it follows that if we take six  of these of the order 
$2$ permutations corresponding to an $E_6$ sub-diagram in $\eu{P}$ then
they generate $\rho(G)$ and  the action of these six generators precisely
recovers the action of the Weyl group of type $E_6$ on 
a $27$ element set (for example the short vectors of $E_6^{\vee}$ modulo sign);
see \cite{Sim}, \cite{HR}.
\end{remark}
\begin{topic}{\bf Acknowledgement.} This paper grew out of some calculations performed 
during the joint work with Daniel Allcock and  Eduard Looijenga on 
finding a presentation of moduli space of smooth cubic surfaces in terms of the Petersen graph
\cite{ABL}.
I would like to thank both of them for many useful conversations and for teaching me a lot
about cubic surfaces. GPT 5.5 was very useful in the final stages of preparing this preprint. It was used extensively  to find optimal versions of some statements, in particular \ref{l-monodromy-of-roots-lemma} and those in 
section \ref{def-names}, 
for creating figures  \ref{figure-Petersen}, \ref{fig-six-Petersens}, and for 
double checking calculations. It helped find and correct many typos.
\end{topic}

\end{document}